\documentclass{amsart}[12pt]
\newtheorem{theorem}{Theorem}

\newtheorem{corollary}[theorem]{Corollary}

\theoremstyle{definition}

\newtheorem{example}[theorem]{Example}

\newtheorem{remark}[theorem]{Remark}

\DeclareMathOperator{\Sym}{Sym}
\DeclareMathOperator{\GL}{GL}

\DeclareMathOperator{\charr}{char\,}
\DeclareMathOperator{\tr}{tr}
\newcommand{\N}{{\mathbb{N}}}

\newcommand{\C}{{\mathbb{C}}}
\newcommand{\op}{\mathrm{op}}

\newcommand{\action}{\!\circ}

\newcommand{\KXd}{K\langle X_d\rangle}

\newcommand{\KSS}{\big(\KXd,\action\big)}
\newcommand{\Cuv}{\C \langle u,v\rangle}
\newcommand{\CuvD}{{\mathbb C}\langle u,v\rangle^{D_{2n}}}
\newcommand{\Dn}{D_{2n}}
\newcommand{\CuvDc}{\big(\C\langle u,v\rangle^{D_{2n}},\circ\big )}

\begin{document}

\title[On dihedral invariants of the free associative algebra of rank two]
{On dihedral invariants of the free associative algebra of rank two}

\author[Silvia Boumova, Vesselin Drensky, {\c S}ehmus F{\i}nd{\i}k]
{Silvia Boumova$^{1}$, Vesselin Drensky$^{2}$, {\c S}ehmus F{\i}nd{\i}k$^{1,3}$}
\address{$^{1}$Faculty of Mathematics and Informatics,
Sofia University, Sofia, Bulgaria and Institute of Mathematics and Informatics,
Bulgarian Academy of Sciences,
1113 Sofia, Bulgaria\\}
\email{boumova@fmi.uni-sofia.bg}
\address{$^{2}$Institute of Mathematics and Informatics,
Bulgarian Academy of Sciences,
1113 Sofia, Bulgaria}
\email{drensky@math.bas.bg}
\address{$^{3}$Department of Mathematics,
\c{C}ukurova University, 01330 Balcal\i,
 Adana, Turkey\\}
\email{sfindik@cu.edu.tr}

\thanks
{The research was partially supported by the Bulgarian Ministry of Education and Science, Scientific Programme ``Enhancing the
Research Capacity in Mathematical Sciences (PIKOM)'', No. DO1-67/05.05.2022}


\begin{abstract}

Let $\KXd$ denote the free associative algebra of rank $d \geq 2$ over a field $K$. By results of Lane (1976) and Kharchenko (1978),
the algebra of invariants $\KXd^G$ is free for any subgroup $G \leq \GL_d(K)$ and any field $K$.
Koryukin (1984) introduced an additional action of the symmetric group $\Sym(n)$ on the homogeneous component of degree $n$ of $\KXd$,
given by permuting the positions of the variables. This endows $\KXd$ with the structure of a $\KSS$-$S$-algebra.
With respect to this action, Koryukin proved that the invariant algebra $\KXd^G$ is finitely generated for every reductive group $G$.

In this paper we study the algebra $\CuvD$ of invariants under the action of the dihedral group $\Dn$ on the free associative algebra $\Cuv$ of rank $2$.
We compute the Hilbert series of $\CuvD$ and construct an explicit set of generators for $\CuvD$ as a free algebra.
Furthermore, we describe a finite generating set for the $S$-algebra $\CuvDc$.
\end{abstract}

\maketitle

\section{\textbf{Introduction}}
Let $K$ be a field of characteristic zero, and let $K[X_d] = K[x_1, \ldots, x_d]$ denote the free associative commutative unital algebra of rank $d > 1$ over $K$.
For a subgroup $G$ of $\GL_d(K)$ the algebra of invariants $K[X_d]^G$ consists of all $G$-invariant polynomials.
The commutative case is well-studied, with foundational results stemming from Hilbert's 14th problem, posed in 1990 at the International Congress of Mathematicians in Paris \cite{Hi}.
In 1916, E. Noether \cite{No1} proved that for any finite group $G$, the algebra $K[X_d]^G$ is finitely generated when $\text{char}(K) = 0$.
In 1926 she extended this result to arbitrary characteristic \cite{No2}.
For reductive groups in characteristic zero, finite generation also follows from Hilbert's work.
However in 1950 Nagata \cite{Na} demonstrated that finite generation fails in general for infinite groups, providing a counterexample.

We now turn to the noncommutative setting, replacing $K[X_d]$ with an algebra sharing its key properties. The natural candidate is the free associative algebra $\KXd$
(i.e., the algebra of noncommuting polynomials in $d$ variables), which satisfies the same universal property.
Remarkably, the analogue of Noether's theorem on finite generation holds for $\KXd^G$ only in a highly restricted case:
for finite subgroups $G \leq \text{GL}_d(K)$, the work of Koryukin \cite{Kor}, Dicks-Formanek \cite{DiFo}, and Kharchenko \cite{Kh2}
established that $\KXd^G$ is finitely generated if and only if $G$ acts by scalar multiplication on the vector space spanned by $X_d$.
Later Koryukin \cite{Kor} generalized this to infinite groups.

In 1984 Koryukin \cite{Kor} introduced an additional action of the symmetric group $\Sym(n)$ on the homogeneous component of degree $n$ of $\KXd$,
given by permuting the positions of the variables. This endows $\KXd$ with the structure of a $\KSS$-$S$-algebra.
With respect to this action, Koryukin proved that the invariant algebra $\KXd^G$ is finitely generated for every reductive group $G$.

One of the natural problems both in the commutative and noncommutative case is the description of the invariants of important groups.
In this paper we study the algebra $\CuvD$ of invariants under the action of the dihedral group $\Dn$ on the free associative algebra $\Cuv$ of rank $2$.
We describe several important properties of $\CuvD$. We compute its Hilbert series and construct a homogeneous set of generators of $\CuvD$ as a free algebra.
Finally, we find a finite generating set of $\CuvD$ as an $S$-algebra.

Similar problems for the noncommutative symmetric polynomials, i.e. for the algebra of invariants $K\langle X_d\rangle^{\Sym(d)}$ were solved in \cite{BDDK}.
For other recent results on noncommutative invariants of the dihedral group see \cite{D4} and \cite{D5}.

\section{\textbf{Preliminaries}}

We assume the ground field to be the field of complex numbers, i.e.,
$K = \C$, although many results hold over any field $K$ of characteristic zero.

\subsection{Classical Commutative Invariant Theory}

The general linear group acts naturally on the $d$-dimensional complex vector space $V_d$ with basis $\{v_1,\ldots,v_d\}$. The coordinate functions $x_i:V_d\to {\mathbb C}$ are defined by
\[
x_i(\xi_1v_1+\cdots+\xi_dv_d)=\xi_i,\quad \xi_1,\ldots,\xi_d\in \C, \quad i=1,\ldots,d.
\]

This induces an action of on the algebra of polynomial functions $\C[X_d]={\mathbb C}[x_1,\ldots,x_d]$ via
\[
g(f):v\to f\big(g^{-1}(v)\big),\quad g\in \GL_d(\C),\, f(X_d)\in \C[X_d],\, v\in V_d.
\]

For a subgroup $G$ of $\GL_d(\C)$, the algebra of $G$-invariants consists of all polynomials fixed under this action:
\[
\C[X_d]^G=\big\{f\in \C[X_d]\mid g(f)=f\text{ for all }g\in G\big\}.
\]

For our purposes it is more convenient to assume that $\GL_d(\C)$ acts canonically on the vector space $KX_d$ with basis $X_d$ and to extend its action diagonally to $\C[X_d]$ by
\[
g\big(f(x_1,\ldots,x_d)\big)=f\big(g(x_1),\ldots,g(x_d)\big),\quad g\in \GL_d(\C), f\in \C[X_d].
\]
The diagonal action of its opposite group $\GL_d^\op(\C)$ induced by the canonical action of $\GL_d^\op(\C)$ on the vector space $KX_d$
is the same as the action of $\GL_d(\C)$ on the polynomial functions $\C[X_d]$.
The mapping $g\to g^{-1}$ defines an isomorphism of $\GL_d(\C)$ and $\GL_d^\op(\C)$ and both actions give the same algebras of invariants.

\subsection{Foundational Theorems}

The following classical theorems form the cornerstone of commutative invariant theory.

\begin{theorem}[\rm (Endlichkeitssatz of Emmy Noether~\cite{No1})]
\label{finite generation Noether}
For a finite subgroup $G$ of $\GL_d(K)$ and a field $K$ of characteristic zero, the algebra of invariants $K[X_d]^G$ is finitely generated.
 Moreover, it admits a system of homogeneous generators whose degrees are bounded by the order of $G$.
\end{theorem}

\begin{theorem}[Chevalley--Shephard--Todd, \cite{Che, ShT}]
\label{pseudoreflections}
For a finite group $G$ over a field $K$ of characteristic zero, the algebra of invariants $K[X_d]^G$ is isomorphic to a polynomial algebra
(i.e., has  a system of algebraically independent generators) if and only if $G< \GL_d(K)$ is generated by pseudo-reflections
(matrices of finite order with a conjugate that is a diagonal matrix of the form $\mathrm{diag}(1, \ldots , 1, \xi)$, where $\xi\neq  1$ is a root of unity).
\end{theorem}

The algebra $K[X_d]^G$ is naturally graded, decomposing into homogeneous components:
\[
K[X_d]^G=K\oplus \left(K[X_d]^G\right)^{(1)} \oplus \left(K[X_d]^G\right)^{(2)} \oplus \cdots.
\]
Its Hilbert (or Poincar\'e) series is given by
\[
H\big(K[X_d]^G,t\big)=\sum_{n\geq 0}\dim \left(K[X_d]^G\right)^{(n)}  t^n
\]
and by the Hilbert–Serre theorem, this series is rational:
\[
H\big(K[X_d]^G,t\big) = p(t) \prod_{i=1}^m\frac{1}{1-t^{a_i}},\quad p(t)\in{\mathbb Z}[t].
\]

The third fundamental result, due to Molien \cite{Mo}, provides an explicit formula for the Hilbert series:

\begin{theorem}\label{Molien formula}
For a finite group and $\charr (K)=0$,
\[
H\big(K[X_d]^G,t\big)=\frac{1}{\vert G\vert}\sum_{g\in G}\frac{1}{\det\big(1-gt\big)}.
\]
\end{theorem}

\subsection{Noncommutative Invariant Theory}

In one of the principal branches of noncommutative invariant theory, we replace the commutative polynomial algebra $\mathbb{C}[X_d]$ with an algebra sharing its key properties.
For comprehensive surveys on this subject, we recommend the works of Formanek \cite{Fo} and the second named author of this paper \cite{Dr1}.

Our primary object of study is the free associative algebra $\KXd$.
Among unital associative algebras, $\KXd$ satisfies the same universal property as $K[X_d]$ does in the commutative case:
any map $X_d \to R$ to an algebra $R$ extends uniquely to a homomorphism $\KXd \to R$. T
he general linear group $\GL_d(K) = \GL(KX_d)$ acts naturally on the vector space $KX_d$ with basis $X_d$, and this action extends to $\KXd$ via algebra homomorphisms:
\[
g(f(x_1,\cdots, x_d )) = f(g(x_1),\cdots, g(x_d)),\, g \in GL_d , f(x_1,\cdots, x_d ) \in \KXd.	
\]

For any subgroup $G \leq \GL_d(K)$, the algebra of $G$-invariants $\KXd^G$ consists of all polynomials fixed by this action.

\subsection{Fundamental Results on Finite Generation}

The following results demonstrate the stark contrast between commutative and noncommutative invariant algebras regarding finite generation.
These were established independently by Dicks-Formanek \cite{DiFo} and Kharchenko \cite{Kh2} for finite groups, and extended by Koryukin \cite{Kor} to arbitrary groups.

\begin{theorem}
\label{noncommutative finite generation}
Let $G \leq \GL_d(K)$ over any field $K$, and let $KY_m$ be the minimal subspace of $KX_d$ with $\KXd^G \subseteq K\langle Y_m\rangle$.
Then the algebra $\KXd^G$ is finitely generated if and only if $G$ acts on $KY_d$ by scalar multiplication.
\end{theorem}

\begin{corollary}[{\cite{DiFo, Kh2}}]
For finite $G \leq \GL_d(K)$, the algebra of invariants $\KXd^G$ is finitely generated if and only if $G$ is cyclic and consists of scalar matrices.
\end{corollary}

\begin{corollary}[{\cite{Kor}}]
	If $G$ acts irreducibly on $KX_d$, then $\KXd^G$ is either trivial or not finitely generated.
\end{corollary}

It has been shown that the counterpart of the Chevalley--Shephard--Todd theorem also looks different for the free associative algebra $K\langle X_d\rangle$.

\begin{theorem}\label{free generation of noncommutative invariants}
{\rm (i) (Lane \cite{La}, Kharchenko \cite{Kh1})} For any $G \leq \GL_d(K)$ and any field $K$, the algebra $\KXd^G$ is free.
		
{\rm (ii) (Kharchenko \cite{Kh1})} For finite groups $G$, there exists a Galois correspondence between free subalgebras of $\KXd$ containing $\KXd^G$ and subgroups of $G$:
a subalgebra $F$ with $\KXd^G \subseteq F$ is free if and only if $F = \KXd^H$ for some $H \leq G$.
\end{theorem}

The analogue of the Molien formula for the Hilbert series of invariants admits the following elegant form:

\begin{theorem} [\label{noncommutatiev Molien formula}{Dicks and Formanek~{\cite{DiFo}}}]
For a finite group $G \leq \GL_d(K)$ with $\mathrm{char}(K) = 0$,
\[
H\big(K\langle X_d\rangle^G,t\big)=\frac{1}{\vert G\vert}\sum_{g\in G}\frac{1}{1-\tr(g)t}.
\]
\end{theorem}

\section{\textbf{The algebra} $\CuvD$}
The canonical action of the dihedral group
\[
D_{2n}=\langle\rho,\tau\mid\rho^n=\tau^2=(\tau\rho)^2=1\rangle
\]
on the two-dimensional vector space is as the group of symmetries of the regular $n$-gon.
Working over $\mathbb C$ we may change the basis, see Klein \cite{Kl} (Kapitel II in the German original or Chapter II in the English translation) and may assume that
it acts on the vector space with basis $\{u,v\}$ via
\[
\begin{matrix}
	\rho:&u\to\xi u\\
	& v\to\bar\xi v
\end{matrix}
\ \ \ \ \ \ \ \  \begin{matrix}
	\tau:&u\to v\\
	&v\to u
\end{matrix}
\]
where $\xi=e^{\frac{2\pi i}{n}}$ is a primitive $n$-th root of unity and $\bar\xi=\xi^{-1}$.
In the sequel we shall consider this action of $D_{2n}$ on $\mathbb{C}\langle u,v \rangle$.

Considering polynomials $p(u,v) \in \mathbb{C}\langle u,v \rangle^{D_{2n}}$ with $p(0,0) = 0$ (omitting constant terms since they are naturally invariant), we observe:
 \begin{enumerate}
 	\item \textbf{Symmetry Condition:}
 \[
 p(u,v)=\tau\cdot p(u,v)=p(v,u).
 \]
 Thus all invariants are symmetric in $u$ and $v$.
\item \textbf{Grading Condition.}
The $\rho$-invariance $p(u,v) =\rho\cdot p(u,v)= p(\xi u, \bar{\xi} v)$ implies for homogeneous in $u$ and $v$ components
\[
w(u,v)+w(v,u)=\xi^{\text{\rm deg}_uw-\text{\rm deg}_vw}w(u,v)+\xi^{\text{\rm deg}_vw-\text{\rm deg}_uw}w(v,u).
\]
Therefore
 \[
 \text{\rm deg}_uw-\text{\rm deg}_vw\equiv0 \ (\text{\rm mod} \ n).
 \]
 \end{enumerate}

\begin{remark}\label{H(F,t) and g(t)}
Let $F$ be a free subalgebra of the free algebra ${\mathbb C}\langle X_d\rangle$ freely generated by a set $Y$ of homogeneous polynomials and $g_n$ be the number of polynomials of degree $n$ in $Y$.
It is well known that the Hilbert series of $F$ can be expressed by the generation function of the sequence $g_1,g_2,\ldots$, as
\[
H(F,t)=\frac{1}{1-G(F,t)},\text{ where }G(F,t)=\sum_{n\geq 1}g_nt^n.
\]
\end{remark}

It is known by \cite{La,Kh1} that $\CuvD$ is a free algebra. Also by \cite{DiFo}, the Hilbert series of the algebra $\CuvD$ is
\[
H_{2n}(t) = H(\CuvD,t)=\frac{1}{2n}\sum_{g\in \Dn}\frac{1}{1-\text{\rm tr}(g)t}
\]
Using the symmetry and grading conditions on the $D_{2n}$-invariants together with Remark \ref{H(F,t) and g(t)}
we shall compute the Hilbert series for the algebra of invariants $\mathbb{C}\langle u,v \rangle^{D_{2n}}$ and  shall present its homogeneous free generating set.

\begin{theorem}
	The Hilbert series $H_{2n}(t)$ of $\mathbb{C}\langle u,v \rangle^{D_{2n}}$ is given by
	\begin{itemize}
		\item [(i)]for odd  $n=2m+1$, $m\geq1$
		\[
		H_{2n}(t)=\frac{1}{2}+\frac{1}{2n(1-2t)}+\frac{1}{n}\sum_{k=1}^{m}\frac{1}{1-2\text{\rm cos}(\frac{2k\pi}{n})t},
		\]
		
		\item [(ii)] for even $n=2m+2$, $m\geq1$
		\[
		H_{2n}(t)=\frac{1}{2}+\frac{1}{2n(1-2t)}+\frac{1}{2n(1+2t)}+\frac{1}{n}\sum_{k=1}^{m}\frac{1}{1-2\text{\rm cos}(\frac{2k\pi}{n})t}.
		\]
	\end{itemize}
\end{theorem}

\begin{proof}
	Consider the canonical matrix representation of the dihedral group. The rotations $r_k$ and the reflections $s_k$ are represented by the matrices
	\[
	r_k=\left(\begin{tabular}{  c  c  }
		$\text{\rm cos}(\frac{2k\pi}{n})$& $-\text{\rm sin}(\frac{2k\pi}{n})$ \\
		&\\
		$\text{\rm sin}(\frac{2k\pi}{n})$ &$\text{\rm cos}(\frac{2k\pi}{n})$\\
	\end{tabular}\right) \ \ \ \text{\rm and} \ \ \
	s_k=\left(\begin{tabular}{  c  c  }
		$\text{\rm cos}(\frac{2k\pi}{n})$& $\text{\rm sin}(\frac{2k\pi}{n})$ \\
		&\\
		$\text{\rm sin}(\frac{2k\pi}{n})$ & $-\text{\rm cos}(\frac{2k\pi}{n})$\\
	\end{tabular}\right).
	\]
Their traces are:
\[
\text{\rm tr}(r_k)=2\text{\rm cos}\left(\frac{2k\pi}{n}\right), \ \   \text{\rm tr}(s_k)=0, \text{ for } k=0,\ldots,n-1.
\]

Hence the Hilbert series of the algebra $\CuvD$ is
\begin{align}
H_{2n}(t)=&\frac{1}{2n}\sum_{g\in D_{2n}}\frac{1}{1-\text{\rm tr}(g)t}\nonumber
		=\frac{1}{2n}\sum_{k=0}^{n-1}\frac{1}{1-\text{\rm tr}(s_k)t}+\frac{1}{2n}\sum_{k=0}^{n-1}\frac{1}{1-\text{\rm tr}(r_k)t}\nonumber\\
		=&\frac{1}{2n}\sum_{k=0}^{n-1}{1}+\frac{1}{2n}\sum_{k=0}^{n-1}\frac{1}{1-2\text{\rm cos}\left(\frac{2k\pi}{n}\right)t}\nonumber\\
		=&\frac{1}{2}+\frac{1}{2n(1-2t)}+\frac{1}{2n}\sum_{k=1}^{n-1}\frac{1}{1-2\text{\rm cos}\left(\frac{2k\pi}{n}\right)t}\nonumber
	\end{align}
	
	(i) Now let $n=2m+1$ for some integer $m\geq1$. Since for $k=1,\ldots,m$
	\[
	\text{\rm cos}\left(\frac{2k\pi}{2m+1}\right)=\text{\rm cos}\left(-\frac{2k\pi}{2m+1}+2\pi\right)=\text{\rm cos}\left(\frac{2(2m+1-k)\pi}{2m+1}\right),
	\]
	we have
\begin{align*}
	H_{2n}(t)=& \frac{1}{2} + \frac{1}{2n(1-2t)}  +\frac{1}{n}\sum_{k=1}^{m}\frac{1}{1-2\cos(\frac{2k\pi}{n})t}	\\
	= & \frac{1}{2} + \frac{1}{2n(1-2t)} -  \frac{1}{n(1-2t)} +  \frac{1}{n(1-2t)} + \frac{1}{n}\sum_{k=1}^{m}\frac{1}{1-2\cos(\frac{2k\pi}{n})t}\\
	= &  \frac{1}{2} - \frac{1}{2n(1-2t)}  + \frac{1}{n}\sum_{k=0}^{m}\frac{1}{1-2\cos(\frac{2k\pi}{n})t}
\end{align*}	
			
	(ii) The proof runs similarly to the proof of (i).
\end{proof}

\subsection{Hilbert series }

Let $n=2m+1$. In the Hilbert series
\begin{align*}
		H_{2n}(t) =\frac{1}{2}+\frac{1}{2n(1-2t)}+\frac{1}{n}\sum_{k=1}^{m}\frac{1}{1-2\text{\rm cos}\left(\frac{2k\pi}{n}\right)t}.
	\end{align*}
we analyze the sum involving cosines
\[
A(t)=\sum_{k=0}^{m}\frac{1}{1-2\cos(\frac{2k\pi}{n})t}
\]
and express it as a rational function. By finding a common denominator, we obtain
\begin{align*}
A(t)= \frac{B(t)}{C(t)},\ \ \text{ where } C(t)=\prod_{k=0}^{m} {\big(1-2\cos(\frac{2k\pi}{n})t \big)} 	
\end{align*}

The denominator $C(t)$ relates to the minimal polynomials of $2\cos(2\pi/n)$. Let
\[
\Psi_n(t)= \prod_{k=0}^{m} {\big(t-2\cos(\frac{2k\pi}{n}) \big)} =\prod_{d|n} \psi_d(t),
\]
 where $\psi_d(t)$ is the minimal polynomial of $2\cos(2\pi/d)$. This gives
\[
C(t)=t^{m+1}\Psi_n(\frac{1}{t}).
\]
The minimal polynomial $\psi_n(t)$ has roots
\[
2\cos(2\pi k/n) \mid 1 \leq k < n/2, \gcd(k,n)=1,
\]
i.e. they are  twice the real parts of the primitive $n$-th roots of unity, and degree $\varphi(n)/2$ for $n \geq 3$ (where $\varphi$ is Euler's totient function).

The relation between $\psi_{n}(x) $ and cyclotomic polynomials $\Phi_{n} $ can be shown in the following identity proved by Lehmer \cite{Leh}, which holds for any non-zero complex number $z$:
\[
\psi_{n}\left(z+z^{-1}\right)=z^{-{\frac {\varphi (n)}{2}}}\Phi _{n}(z).
\]

In 1993, Watkins and Zeitlin \cite{WZ} found the following relation between $\psi_{n}(t)$ and Chebyshev polynomials of the first kind.
\begin{theorem} [ \cite{WZ}]
	For the Chebyshev polynomial $T_k(t)$ of the first kind:
	\begin{enumerate}
		\item If $n = 2m + 1$ is odd, then $ \prod _{d\mid n}\psi _{d}(2t)=2(T_{m+1}(t)-T_{m}(t))$;
		
		\item if $n = 2 m$  is even, then $\prod _{d\mid n}\psi _{d}(2t)=2(T_{m+1}(t)-T_{m-1}(t))$;
		
		\item 	if $n$ is a power of 2, i.e. $n=2^k$ we have  $\psi _{2^{k+1}}(2t)=2T_{2^{k-1}}(t)$.
	\end{enumerate}
	
\end{theorem}

From the preceding theorem, we obtain the factorization
\[
\chi_n(t)= \prod_{k=0}^{m} {\big(t-2\cos(\frac{2k\pi}{n}) \big)} = \prod _{d\mid n}\Psi _{d}(t)
\]

The numerator $B(t)$ of $A(t)$ admits an explicit expansion in terms of elementary symmetric polynomials
\[
B(t)=  \sum_{s=0}^{m} \prod_{\substack{k=0,\\ k\neq s}}^{m} {\left(1-2\cos\left(\frac{2k\pi}{n}\right)t \right)}.
\]

This can be systematically expanded as:
\begin{align*}
	B(t) &= 1 - t\left(\sum_{s=0}^{m}\sum_{\substack{k=0 \\ k\neq s}}^{m} 2\cos\left(\frac{2k\pi}{n}\right)\right) \\
   	   	&\quad +t^2\left( \sum_{\substack{ \text{subsets } (\beta_1,\dots \beta_m) \text{ of }  m \\ \text{ among } (m+1) \text{ variables}\
   \ \{2\cos(\frac{0\pi}{n}), \cdots, 2\cos(\frac{2m\pi}{n})\}} } \sigma_2^{m}\big(\beta_1,\dots \beta_m \big) \right)\\
		&\quad - \cdots \\
		&\quad + (-1)^m t^m \sigma_m\left(2\cos\left(\frac{0\pi}{n}\right), \ldots, 2\cos\left(\frac{2m\pi}{n}\right)\right)
\end{align*}
where $\sigma_k$ denotes the $k$-th elementary symmetric polynomial in all $(m+1)$ variables,
namely $\{2\cos(\frac{0\pi}{n}), \cdots, 2\cos(\frac{2m\pi}{n})\}$,
and $\sigma_k^m$ denotes the symmetric polynomial in subsets of $m$ variables among all. The coefficients follow the pattern
\[
B(t) = (m+1) -  m\sigma_1 t + (m-1)\sigma_2 t^2 -\cdots +(-1)^{m-1}  2\sigma_{m-1} t^{m-1}+(-1)^m\sigma_m t^m.
\]

For the even case $n = 2m$, the Hilbert series takes the form:
\begin{align*}
	H_{4m+4}(t) &= \frac{1}{2} + \frac{1}{4m+4}\left(\frac{1}{1-2t} + \frac{1}{1+2t}\right) + \frac{1}{2m+2}\sum_{k=1}^{m}\frac{1}{1-2\cos\left(\frac{2k\pi}{n}\right)t} \\
	&= \frac{1}{2} - \frac{1}{4m+4}\left(\frac{1}{1-2t} + \frac{1}{1+2t}\right) + \frac{1}{2m+2}A(t),
\end{align*}
where $A(t) = \frac{B(t)}{C(t)}$ as before. The numerator  maintains the same symmetric structure
\[
B(t)= (m+1) -  m\sigma_1 t +\cdots +(-1)^{m-1}  2\sigma_{m-1} t^{m-1}+(-1)^m \sigma_m t^m,
\]
where $\sigma_k$ are the $k$-th elementary symmetric polynomial in all $(m+1)$ variables  $\{2\cos\left(\frac{2k\pi}{n}\right)\mid k = 0,\ldots,m\}$.

\subsection{Linear Recurrence Relations and Rational Generating Functions}
A fundamental result in combinatorics and number theory establishes the correspondence between linear recurrence sequences and rational generating functions.
Let us consider a sequence of complex numbers $(R_n)_{n=0}^\infty$ satisfying a $k$-th order linear recurrence relation:
\[
R_n = \sum_{j=1}^{k} A_j R_{n-j} \text{ for } n\geq k,
\]
where $A_k \neq 0$, and the initial values $R_0, \ldots, R_{k-1}$ are given complex numbers (not all zero). The associated characteristic polynomial is
\[
c(x) = x^k - \sum_{j=1}^{k}A_j x^{k-j}.
\]
The generating function of the sequence is defined as the formal power series
\[
R(x) = \sum_{n=0}^{\infty}R_nx^n.
\]
The following theorem establishes the fundamental connection between these concepts
\begin{theorem}[{\cite{JoKiss}}]\label{recseq}
	Let $(a_n)_{n=0}^\infty$ be a sequence satisfying a linear recurrence relation with characteristic polynomial $x^k - \sum{j=1}^k A_j x^{k-j}$.
Then its generating function is rational and takes the form
\[	
a(x) =\frac{\sum_{i=0}^{k-1}b_ix^i}{1 - \sum_{j=1}^{k}A_j x^j},
\]
where the coefficients $b_i$ are determined by:	
\[	
b_i = a_i - \sum_{j=1}^{i}	A_j a_{i-j} \ \text{ for }	\ 0\leq i \leq k - 1 ).
\]

Conversely, any rational function
\[	
\frac{\sum_{i=0}^{s}b_ix^i}{\sum_{j=0}^{k}d_jx^j} \quad (s\geq0, k\geq 1,b_s\neq 0, d_0 \neq 0, d_k\neq 0),
\]
with denominator not dividing the numerator, generates a sequence satisfying a linear recurrence relation with characteristic polynomial
\[	
x^k +\sum_{j=1}^{k}\frac{d_j}{d_0}x^{k-j}
\]
for all $n \geq \max(k, s+1)$. When $s < k$, $d_0 = 1$, and all coefficients are rational integers, the sequence terms $a_n$ are rational integers.
\end{theorem}

Using Theorem \ref{recseq} we summarize the results for the Hilbert series $H_{2n}(t)$ for small $n$ in the table:
\begin{table}[h!]
	\centering
	\begin{tabular}{|c|c|c|}
		\hline
		$n$ & $H_{2n}(t)$ & \textbf{Recurrence relation} \\
		\hline \hline
		$3$ & $\frac{t^2+t-1}{2t^2+t-1}$ & $a_{m+2} = a_{m+1} + 2a_m$ \\ [0.2cm]
		\hline
		$4$ & $\frac{1-3t^2}{1-4t^2}$  & $a_{m+2} = 4a_m$ \\ [0.2cm]
		\hline
		5 &  $\frac{t^3-2t^2-t+1}{2t^3-3t^2-t+1}$& $a_{m+3} = 3a_{m+2} - 3a_{m+1} + 2a_m$ \\ [0.2cm]
		\hline
		6 & $\frac{ 2t^4-4t^2+1}{4t^4-5t^2+1}$ & $a_{m+4} = 5a_{m+2}- 4a_m$ \\ [0.2cm]
		\hline
		7 &  $\frac{t^4+2t^3-3t^2-t+1}{2t^4+3t^3-4t^2-t+1}$ & $a_{m+4} = a_{m+3} + 4a_{m+2} - 3a_{m+1} - 2a_m$ \\ [0.2cm]
		\hline
		8 & $\frac{5t^4-5t^2+1}{8t^4-6t^2+1}$  & $a_{m+4} = 6a_{m+2} - 8a_m$\\ [0.2cm]
		\hline
		9 & $\frac{t^5-3t^4-3t^3+4t^2+t-1}{2t^5-5t^4-4t^3+5t^2+t-1}$   & {\small $a_{m+5} = a_{m+4} + 5a_{m+3} - 4a_{m+2} -5a_{m+1} + 2a_m$} \\ [0.2cm]
		\hline 	
		10 & $\frac{2t^6-9t^4+6t^2-1}{4t^6-13t^4+7t^2-1}$  & $a_{m+6} = 7a_{m+4} - 13a_{m+2} + 4a_m$ \\ [0.2cm]
		\hline
	\end{tabular}
	\caption{Hilbert series for small $n$ and recurrence relation}
\end{table}

Using Remark \ref{H(F,t) and g(t)} the knowledge of the Hilbert series $H_{2n}(t)$ allows to find the generating function $G_{2n}(t)$
for the free generators of the algebra ${\mathbb C}\langle u,v\rangle^{D_{2n}}$. The results for small $n$ are given in the table:
\begin{table}[h!]
	\centering
	\begin{tabular}{|c|c|c|}
		\hline
		$n$ & $G_{2n}(t)$ & \textbf{Recurrence relation} \\
		\hline \hline
		$3$ & $\frac{-t^2}{t^2+t-1}$ & $a_{m+2} = a_{m+1} + a_m$ \\ [0.2cm]
		\hline
		$4$ & $\frac{t^2}{1-3t^2}$  & $a_{m+2} = 3a_m$ \\ [0.2cm]
		\hline
		5 &  $\frac{-t^3+t^2}{t^3-2t^2-t+1}$& $a_{m+3} = a_{m+2} +2a_{m+1} -a_m$ \\ [0.2cm]
		\hline
		6 & $\frac{ -2t^4+t^2}{2t^4-4t^2+1}$ & $a_{m+4} = 4a_{m+2}- 2a_m$ \\ [0.2cm]
		\hline
		7 &  $\frac{-t^4-t^3+t^2}{t^4+2t^3-3t^2-t+1}$ & $a_{m+4} = a_{m+3} + 3a_{m+2} - 2a_{m+1} - a_m$ \\ [0.2cm]
		\hline
		8 & $\frac{-3t^4+t^2}{5t^4-5t^2+1}$  & $a_{m+4} = 5a_{m+2} - 5a_m$\\ [0.2cm]
		\hline
		9 & $\frac{-t^5+2t^4+t^3-t^2}{t^5-3t^4-3t^3+4t^2+t-1}$   & {\small $a_{m+5} = a_{m+4} + 4a_{m+3} - 3a_{m+2} -3a_{m+1} + a_m$} \\ [0.2cm]
		\hline 	
		10 & $\frac{-2t^6+4t^4-t^2}{2t^6-9t^4+6t^2-1}$  & $a_{m+6} = 6a_{m+4} - 9a_{m+2} + 2a_m$ \\ [0.2cm]
		\hline
	\end{tabular}
	\caption{Generation functions for small $n$ and recurrence relation}
\end{table}

\subsection{Set of generators for ${\mathbb C}\langle u,v\rangle^{D_6}$}

We consider the case for $n=3$.
We denote by $w_k$ the leading term of a homogeneous element $w\in{\mathbb C}\langle u,v\rangle^{D_6}$ of degree $k$.
The generating function
\[
G_6(t)=G({\mathbb C}\langle u,v\rangle^{D_6},t)=\sum_{k\geq 2}g_kt^k
\]
for the homogeneous set of free generators of ${\mathbb C}\langle u,v\rangle^{D_6}$ is
is given by
\[
G_6(t) = \frac{-t^2}{t^2+t-1} = t^2 + t^3 + 2t^4 + 3t^5 + 5t^6 + 8t^7 + 13t^8 + 21t^9 + 34t^{10} + 55t^{11} + \dots 	
\]
and corresponding relation is
\[
a_{m+2} = a_{m+1} + a_m=2a_{m}+a_{m-1}, \ \text{ with } a_0=a_1=0.
\]
In other words, the number of free generators of degree $k+2$ is equal to the $k$-th Fibonacci number, $k=0,1,2,\ldots$.
Applying the recurrence formula, we determine the number of polynomials of degree $k$, given by $a_{k+2}=2a_{k}+a_{k-1}$. The leading terms of the generating set are constructed as follows:
\[
\{ \ (uv)w_{k}, \ w_{k}(uv),\ (u^3 )w_{k-1}\  \}.
\]
It follows that the number of such terms is the sum of the twice number of $w_{k}$ and  the number of $w_{k-1}$.

\begin{table}[h!]
	\label{T4}
	\centering
	\begin{tabular}{|c|c|c|}
		\hline
		$d$ & no & set of LT of generators of degree $d$ \\
		\hline \hline
		1 & 0 & -\\ [0.2cm]
		\hline
		2 & 1  & \{ $uv$ \} \\ [0.2cm]
		\hline
		3 &  1 & \{ $u^3$ \} \\ [0.2cm]
		\hline
		4 & 2 & \{ $u^2v^2, \ uv^2u\ $ \} \\ [0.2cm]
		\hline
		5 & 3  &\{ $ u^4v,   u^2vu^2,  uv^4\ $ \}\\ [0.2cm]
		\hline
		6 & 5  & \{ $ (uv)w_4, w_4(uv), (u^3)w_3$  \}\\ [0.2cm]
		\hline 	
			& & $\dots$ \\ [0.2cm]
		\hline
		$k $  &  &\{ $ (uv)w_{k}, \ w_{k}(uv),\ (u^3 )w_{k-1}  $ \} \\ [0.2cm]
		\hline 	
	\end{tabular}
\caption{Set of leading terms of generators of degree $d$}
\end{table}

The Hilbert series of ${\mathbb C}\langle u,v\rangle^{D_6}$ is given by
\[
H_6(t)=\frac{t^2+t-1}{2t^2+t-1} = 1 + t^2 + t^3 + 3t^4 + 5t^5 + 11t^6 + 21t^7 + 43t^8 + 85t^9 + \dots
\]
and the coefficients satisfy the recurrence relation
\[
a_{m+2} = a_{m+1} + 2a_m = a_m+2a_{m-1} + 2a_m = 3a_m + 2a_{m-1}, \ \text{ with } a_0=1, a_1=0.
\]
The algebra ${\mathbb C}\langle u,v\rangle^{D_6}$ has a vector space basis consisting of monomials in its free generators.
The basis element of degree 2 is $uv+vu$ and its leading term is $uv$.
The basis element $u^3+v^3$ of degree 3 gives the leading term $u^3$. For degree 4, the leading terms are $\{ u^2v^2, uvuv, uv^2u \}$.
Applying the recurrence formula, we determine the number of polynomials of degree $k+2$, given by $a_{k+2}=3a_{k}+2a_{k-1}$.
The leading terms of the elements of the basis are constructed as follows:
\[
\{ \  (uv)w_{k}, w_{k}(uv), w_{k}(vu), (u^3 )w_{k-1}, w_{k-1}(v^3)  \  \}.
\]
From this construction, it follows that the number of such terms is the sum of three times the number of $w_{k}$ and twice the number of $w_{k-1}$.

\begin{table}[h!]
	\centering
	\begin{tabular}{|c|c|c|}
		\hline
		degree $d$ & number & set of LT of basis elements of degree $d$   \\
		\hline \hline
		0 & 0 & -\\ [0.2cm]
		\hline
		1 & 0 & -\\ [0.2cm]
		\hline
		2 & 1  & \{ $uv$ \} \\ [0.2cm]
		\hline
		3 &  1 & \{ $u^3$ \} \\ [0.2cm]
		\hline
		4 & 3 & \{ $u^2v^2, \ uvuv, \ uv^2u\ $ \} \\ [0.2cm]
		\hline
		5 & 5  &\{ $    u^4v, u^3vu, u^2vu^2, uvu^3, uv^4 $ \}\\ [0.2cm]
		\hline
		6 & 11  &\{ $ (uv)w_4, \ w_4(uv), \ w_4(vu), \  (u^3 )w_3, \ w_3(v^3) $ \}\\ [0.2cm]
		\hline
		& & $\dots$ \\ [0.2cm]
		\hline
		$k+2 $ &    & \{ $ {\tiny   (uv)w_{k }, \ w_{k }(uv), \ w_{k}(vu), \ (u^3 )w_{k-1}, \ w_{k-1}(v^3) } $ \}\\ [0.2cm]
		\hline
	\end{tabular}
	\caption{Set of leading terms of elements of degree $d$ in the basis}
\end{table}

\section{\textbf{The $S$-algebra} $(\CuvD,\circ)$}

Let ${\mathbb C}\langle X_d\rangle^{(k)}$ denote the homogeneous subspace of total degree $k$ in the free associative algebra ${\mathbb C}\langle X_d\rangle$.
The symmetric group $\Sym(d)$ acts on monomials in ${\mathbb C}\langle X_d\rangle^{(k)}$ via index permutation
\[
\pi(x_{i_1}\cdots x_{i_k})=x_{\pi(i_1)}\cdots x_{\pi(i_k)} \ , \ \ \pi\in \Sym(d).
\]
Koryukin \cite{Kor} introduced a distinct $\Sym(k)$-action that permutes variable positions
\[
(x_{i_1}\cdots x_{i_k})\circ \pi=x_{i_{\pi^{-1}(1)}}\cdots x_{i_{\pi^{-1}(k)}}.
\]
This action differs fundamentally from the standard action, as it rearranges positions rather than exchanging generators. For example:
\begin{align*}
	(x_1x_3x_6)\circ (13)&=x_6x_3x_1,\\
	(13)(x_1x_3x_6)&=x_3x_1x_6.
\end{align*}

The algebra $({\mathbb C}\langle X_d\rangle,\circ)$ is equipped with this additional $\Sym(k)$-action on each homogeneous component. A graded subalgebra $F$ is called an $S$-algebra, denoted $(F,\circ)$, if it satisfies $F^{(k)} \circ \Sym(k) = F^{(k)}$ for all $k$.

An $S$-algebra $(F,\circ)$ is finitely generated if there exists a finite set $U$ of homogeneous polynomials that generates $F$ under the $S$-algebra operations. Koryukin's results \cite{Kor} establish that $(\mathbb C\langle X_d\rangle^{D_{2n}},\circ)$ is finitely generated.

We define the following important invariant elements
\begin{align*}
	s_a       & = u^a v^a+v^a u^a, \quad a=1\dots n,\\
	p_{kn}    & =u^{kn}+v^{kn}, \quad k\in \N,\\
	p_{(a,b)} & =u^{a}v^{b}+v^{b}u^{a}.
\end{align*}

\begin{theorem}
	The $S$-algebra $\CuvD$ is generated (as an $S$-algebra) by $uv+vu$ and $u^n+v^n$.
\end{theorem}

\begin{proof}

\noindent\textbf{Step 1: Structure of $S$-algebra ($\CuvD$,$\circ$).} The $S$-algebra ${\mathbb C}\langle u,v\rangle^{D_{2n}}$ is spanned by
\[
s_a^{\sigma}=(u^av^a+v^au^a)^{\sigma},\quad a=1,2,\ldots,\sigma\in \Sym(a+b),
\]
\[
p_{(b,c)}^{\tau}=(u^bv^c+v^bu^c)^{\tau}, \quad b-c\equiv 0\text{ (mod $n$)},\tau\in \Sym(b+c),
\]
i.e. it is generated as an $S$-algebra by
\[
s_a=u^av^a+v^au^a, \quad a=1,2,\ldots,
\]
\[
p_{(b,c)}=u^bv^c+v^bu^c, \quad b-c\equiv 0\text{ (mod $n$)}.
\]
	
\noindent\textbf{Step 2: Reduction to Special Case.} Let us consider the $S$-subalgebra $A$ of ${\mathbb C}\langle u,v\rangle^{D_{2n}}$ generated by
\[
s_a,p_{(b,c)},\quad a,b,c\equiv 0\text{ (mod $n$)}.
\]
We introduce new variables $U=u^n,V=v^n$ and consider the $S$-subalgebra $B={\mathbb C}\langle U,V\rangle^{\Sym(2)}$ of ${\mathbb C}\langle U,V\rangle$
of the symmetric polynomials with respect to $U,V$. By the main result of \cite{BDDK} the $S$-algebra $B$ is generated by $U+V$ and $UV+VU$.
Considering ${\mathbb C}\langle U,V\rangle$ as a subalgebra of ${\mathbb C}\langle u,v\rangle$, we obtain that $B$ is a subalgebra of $A$.
It has the property that all elements of $B$ are linear combinations of monomials, where $u$ and $v$ participate consequently multiple of $n$ times.
Applying the action of the symmetric groups $\Sym(nk)$ on the corresponding homogeneous component, $k=1,2,\ldots$,
we obtain that the $S$-subalgebra $B$ of ${\mathbb C}\langle u,v\rangle$
is generated by $s_n$ and $p_n=p_{(n,0)}$.
	
\noindent\textbf{Step 3: Induction.}
 We shall use the idea of the first proof  for $K\langle x_1,x_2\rangle^{\Sym(2)}$ given in \cite[Remark 4.6]{BDDK}.
Let $D$ be the $S$-subalgebra of $\CuvD$ generated by $s_1$ and $p_n$.
The proof of the theorem will be completed
if we show that $p_{(b,c)}\in D$ for all $b,c$ such that
$b\equiv c\text{ (mod $n$)}$.
Thus we may express
\[
p_{(b,c)}=p_{(a+kn,a+ln)}
\]
for $a,k,l\geq0$. We shall apply an induction on $a\geq0$. Initially we observe the following.

\

\noindent Observation: Let $W(u,v)=u^{b-1}v^c$. Then
\begin{align}
	s_1p_{(b,c)}&=(uv+vu)(u^bv^c+v^bu^c)=(uvu^bv^c+vuv^bu^c)+(vu^{b+1}v^c+uv^{b+1}u^c)\nonumber\\
	&=(uvuW(u,v)+vuvW(v,u))+(vuuW(u,v)+uvvW(v,u))\in D,\nonumber
\end{align}
\[
(s_1p_{(b,c)})\circ(13)=(uvuW(u,v)+vuvW(v,u))+(uuvW(u,v)+vvuW(v,u))\in D,
\]
\[
(s_1p_{(b,c)})\circ(23)=(uuvW(u,v)+vvuW(v,u))+(vuuW(u,v)+uvvW(v,u))\in D,
\]
These three equalities imply that
\[
uvuW(u,v)+vuvW(v,u)=uvu^bv^c+vuv^bu^c\in D,
\]
\[
(uvu^bv^c+vuv^bu^c)\circ (2,b+2)=p_{(b+1,c+1)}\in D.
\]

\noindent Thus, we obtain $s_1s_i=s_1p_{(i,i)}\in D$ implies $s_{i+1}\in D$ for $i=1,\ldots,n-1$, successively;
that yields $s_n\in D$. Hence, in the second step of the proof we established that $p_{(k_0n,l_0n)}\in D$.
Therefore, it is sufficient to show that $p_{(b,c)}=p_{(a+kn,a+ln)}\in D$ implies $p_{(b+1,c+1)}=p_{(a+1+kn,a+1+ln)}\in D$,
which is a direct consequence of the observation above.

\end{proof}

\begin{example}
	Consider the element $u^4v+v^4u\in\Cuv^{D_{6}}$. Then we have the following.
	\begin{align}
		\frac{1}{2}\big((u^3+v^3)(uv+vu)\big)\circ(14)&=\frac{1}{2}(u^4v+v^4u+v^3uv+u^3vu)\circ(14)\nonumber\\
		&=\frac{1}{2}(u^4v+v^4u+uv^3v+vu^3u)\nonumber\\
		&=u^4v+v^4u.\nonumber
	\end{align}
\end{example}

\section*{Acknowledgements}

This study of the first named author was partially supported  by the European Union-NextGenerationEU,
through the National Recovery and Resilience Plan of the Republic of Bulgaria, project No. BG-RRP-2.004-0008-C01.

The third named author is very thankful to the Faculty of Mathematics and Informatics of
Sofia University and the Institute of Mathematics and Informatics of
the Bulgarian Academy of Sciences for the creative atmosphere and the warm hospitality during his visit
as a post-doctoral fellow when this project was carried out.


\end{document}